\documentclass[]{amsart}
\usepackage{amsmath}
\usepackage[english]{babel}
\usepackage[utf8]{inputenc}
\usepackage{amsfonts}
\usepackage{amsthm}
\usepackage{esint}

\usepackage{color}
\numberwithin{equation}{section}

\newtheorem{thm}{Theorem}[section]

\newtheorem{prop}[thm]{Proposition}
\newtheorem{cor}[thm]{Corollary}
\newtheorem{lem}[thm]{Lemma}

\theoremstyle{remark}
\newtheorem{rmk}[thm]{Remark}
\theoremstyle{definition}
\newtheorem{defn}{Definition}[section]
\newtheorem{assp}{Assumption}[section]

\DeclareMathOperator{\E}{\mathbb{E}}
\DeclareMathOperator{\BF}{\mathcal{BF}}

\DeclareMathOperator{\SBF}{\mathcal{SBF}}

\DeclareMathOperator{\N}{\mathbb{N}}
\DeclareMathOperator{\R}{\mathbb{R}}

\DeclareMathOperator{\cI}{\mathcal{I}}

\DeclareMathOperator{\cL}{\mathcal{L}}

\DeclareMathOperator{\cB}{\mathcal{B}}

\DeclareMathOperator{\fe}{\mathfrak{e}}

\DeclareMathOperator{\bP}{\mathbb{P}}

\newcommand{\der}[2]{\frac{d #1}{d #2}}
\newcommand{\dersup}[3]{\frac{d^{#3} #1}{d #2^{#3}}}

\newcommand{\Norm}[2]{\left\Vert #1 \right\Vert_{#2}}

\title{Abstract Cauchy problems for the generalized fractional calculus}
\author{Giacomo Ascione}
\address{Dipartimento di Matematica e Applicazioni "Renato Caccioppoli", Università degli Studi di Napoli Federico II}
\email{giacomo.ascione@unina.it}
\begin{document}
	\maketitle
	\begin{abstract}
	We focus on eventually non-linear abstract Cauchy problems with a generalized fractional derivative in time. First we prove a local existence and uniqueness result, then we focus on a generalized Gr\"onwall inequality. Before addressing the inequality, we study some properties of eigenvalues and eigenfunctions of the generalized fractional derivatives. Finally, we prove some consequences of the generalized Gr\"onwall inequality.
	\end{abstract}
\keywords{Keywords: Bernstein functions, Gr\"onwall inequality, contraction theorem, inverse subordinator}
\tableofcontents
	\section{Introduction}
	Fractional calculus has been introduced in the context of the symbolic method for classical calculus by Leibniz himself (see \cite{debnath2004brief}); however we have to wait for contemporary times to see the explosion of studies on fractional derivatives, covering a wide range of applications (see \cite{debnath2003recent}). Together with the classical well-known contributions of Riemann and Liouville, a widely used class of fractional derivatives is given by the Caputo derivatives introduced in \cite{caputo1967linear}, which is simpler to use in the context of applied sciences due to the possibility of defining Cauchy problems with them (see \cite{diethelm2010analysis}).\\
	Different kinds of fractional derivatives were also introduced, starting from the Riemann-Liouville or the Caputo derivative and changing the memory kernel. This is the case, for instance, of the tempered fractional calculus (see \cite{sabzikar2015tempered}). Such kind of operators naturally arise in the theory of generalized fractional calculus, introduced in \cite{kochubei2011general}. In particular, in such work, Caputo-type fractional derivatives were defined by means of complete Bernstein functions, while in \cite{toaldo2015convolution} they were extended to any Bernstein function.\\
	A strict link between Bernstein functions and stochastic process is known: Bernstein functions are Laplace exponents of subordinators, i.e. non-decreasing L\'evy processes (see \cite{bertoin1996levy,schilling2012bernstein}). This link played a crucial role in the classical fractional calculus (see \cite{meerschaert2011stochastic}). Moreover, the linear theory for generalized fractional differential problems has been widely studied in terms of time-changed semigroups by using such stochastic representation (see \cite{toaldo2015convolution,chen2017time}). On the other hand, a theory for non-linear generalized fractional Cauchy problems  is developed (up to our knowledge) only in specific cases, as for instance, in the classical fractional calculus (see \cite{kilbas2006theory,yong2016basic}) and for the tempered fractional calculus (see \cite{li2019well}). \\
	In this paper we consider (eventually non-linear) abstract Cauchy problems with a generalized fractional derivative in time. In particular we focus on the class of special Bernstein functions, of which inverse operators to Caputo-type generalized fractional derivatives are known (see \cite{meerschaert2019relaxation}). Indeed, after Section \ref{Sec2}, in which we introduce the preliminary relation, in Section \ref{Sec3} we prove a local existence and uniqueness result for the aforementioned Cauchy problems under suitable \textit{light} hypotheses on the non-linearity. The fact we are considering a special Bernstein function (with some technical but not so restrictive assumptions) and a quite general non-linearity gives us a unified result that covers different cases, as for instance the tempered one. Moreover, we prove, for example purpose, that we cover the linear case whenever the considered linear operator is continuous. In the linear case we are also able to prove by hand a global uniqueness result (conditioned to global existence of the solutions).\\
	In the general theory of ordinary differential equations a main role is played by the Gr\"onwall inequality (see \cite{ames1997inequalities}). As for existence and uniqueness results, also the Gr\"onwall inequality has been extended to classical and generalized fractional calculus only on specific cases (but with similar techniques): for instance we have generalizations of the Gr\"onwall inequality to the Caputo derivative (see \cite{ye2007generalized}), Caputo-Katugampola derivative (see \cite{almeida2017gronwall}), Caputo derivatives with respect to other functions (see \cite{almeida2018extension}), and also Hadamard derivatives and general $\Psi$-Hilfer operators (see the references in \cite{almeida2018extension}). Here we aim to give a unified proof of a Gr\"onwall inequality for the generalized fractional calculus (as $\Phi$ is a special Bernstein function). To do this, we first need to study eigenvalues and eigenfunctions of the generalized fractional derivatives. This has been done, for negative eigenvalues, in \cite{kochubei2011general} for complete Bernstein functions and in \cite{meerschaert2019relaxation} for special Bernstein functions, while for positive eigenvalues we refer to \cite{kochubei2019growth} in the case of complete Bernstein functions.\\
	In Section \ref{Sec4}, by using the link between Bernstein functions and subordinators, we give a representation of the eigenfunctions of generalized fractional derivative for any (non necessarily special) Bernstein function (under suitable but not restrictive hypotheses) for any eigenvalue. Then, in the case of special Bernstein functions, we express a series representation of such functions. Such functions are then used in Section \ref{Sec5} to prove a Gr\"onwall inequality for the generalized fractional calculus. Finally, in Section \ref{Sec6}, we focus on consequences of this inequality, such as continuous dependence on initial data or parameters and global uniqueness results (under a global existence hypothesis) also in the non-linear case.
	\section{Preliminaries and notation}\label{Sec2}
	In this section we will give some preliminary definitions and fix the notations for what follows.
	\begin{defn}
		We say a function $\Phi \in C^\infty(0,+\infty)$ is a Bernstein function (see \cite{schilling2012bernstein}) if and only if, $\Phi(\lambda) \ge 0$ and, for any $n \in \N=\{1,2,\dots\}$, it holds
		\begin{equation*}
		(-1)^n\dersup{\Phi}{\lambda}{n} \le 0.
		\end{equation*}
		The convex cone (see \cite[Corollary $3.8$]{schilling2012bernstein}) of Bernstein functions will be denoted as $\BF$.
	\end{defn}
	For any Bernstein function, the following theorem, known as L\'evy-Kintchine representation theorem, holds (see \cite[Theorem $3.2$]{schilling2012bernstein}).
	\begin{thm}
		A function $\Phi: (0,+\infty)\to \R$ belongs to $\BF$ if and only if there exist two constants $a_\Phi,b_\Phi \ge 0$ and a measure $\nu_\Phi$ on $(0,+\infty)$ such that
		\begin{equation}\label{intcont}
		\int_{0}^{+\infty}(1 \wedge x)\nu_\Phi(dx)<+\infty,
		\end{equation}
		and
		\begin{equation}\label{eq:LKrepr}
		\Phi(\lambda)=a_\Phi+b_\Phi\lambda+\int_0^{+\infty}(1-e^{-\lambda x})\nu_\Phi(dx).
		\end{equation}
		The measure $\nu$ is called the L\'evy measure of $\Phi$. Vice verse, any triple $(a_\Phi,b_\Phi,\nu_\Phi)$, where $a_\Phi,b_\Phi \ge 0$ and $\nu_\Phi$ is a measure on $(0,+\infty)$ satisfying condition \eqref{intcont}, defines a unique $\Phi \in \BF$ via equation \eqref{eq:LKrepr}.
	\end{thm}
We will denote $\bar{\nu}_\Phi(t)=\nu_\Phi(t,+\infty)$.\\
In general, Bernstein functions can be seen as Laplace exponents of particular L\'evy processes.  Let us recall the following definition.
	\begin{defn}
		A subordinator $\sigma(t)$ (see \cite[Chapter $III$]{bertoin1996levy}) is a non-decreasing L\'evy process.
	\end{defn}
	Concerning the link between subordinators and Bernstein functions, we have the following Theorem (see \cite[Theorem $5.1$]{schilling2012bernstein}).
	\begin{thm}
		For any Bernstein function $\Phi$ there exists a unique subordinator $\sigma_\Phi(t)$ such that 
		\begin{equation}\label{Lapexp}
		\E[e^{-\lambda \sigma_\Phi(t)}]=e^{-t\Phi(\lambda)}.
		\end{equation}
		Viceversa, for any subordinator $\sigma(t)$ there exists a Bernstein function $\Phi_\sigma(t)$ such that Equation \eqref{Lapexp} holds.
	\end{thm}
	Since we focus on Bernstein functions, then we will usually denote subordinators by $\sigma_\Phi(t)$, referring to the fact that its Laplace exponent is given by $\Phi$. Moreover, for any subordinator, the following occupation measure can be defined.
	\begin{defn}
		Let $\sigma_\Phi(t)$ be a subordinator. The potential measure of $\sigma_\Phi(t)$ on $[0,+\infty)$ is defined as
		\begin{equation*}
		U_\Phi(A)=\E\left[\int_0^t \mathbf{1}_{A}(X_s)ds\right], A \in \cB([0,+\infty))
		\end{equation*}
		where $\cB([0,+\infty))$ is the Borel $\sigma$-algebra of $[0,+\infty)$ and $\mathbf{1}_A$ is the indicator function of the Borel set $A$. We will denote the distribution function of the potential measure $U_\Phi(t):=U_\Phi(0,t)$ and we will usually refer to it directly as potential measure.
	\end{defn}
	Actually, we can define a right-continuous inverse for the process $\sigma_\Phi(t)$.
	\begin{defn}
		Let $\sigma_\Phi$ be a subordinator. For any $y \ge 0$ we define
		\begin{equation*}
		L_\Phi(t)=\inf\{y>0: \ \sigma_\Phi(y)>t\},
		\end{equation*}
		called inverse subordinator.
	\end{defn}
	\begin{rmk}
		As shown in \cite{meerschaert2008triangular}, if $\nu_\Phi(0,+\infty)=+\infty$, then $L_\Phi(t)$ is an absolutely continuous random variable for any $t>0$. Let us denote by $f_\Phi(s;t)$ its probability density function.
	\end{rmk}
Once we have defined the inverse subordinator, by using the fact that $\sigma_\Phi$ and $L_\Phi$ are increasing, we have
\begin{equation*}
U_\Phi(t)=\E[L_\Phi(t)].
\end{equation*}
Let us denote by $\cL_{t \to z}$ the Laplace transform acting on the variable $t$, i.e. for any function $f \in L^1_{\rm loc}([0,+\infty),X)$, where $(X, |\cdot|)$ is a Banach space, we consider
\begin{equation*}
\cL_{t \to z}[f(t)](z)=\int_0^{+\infty}e^{-zt}f(t)dt
\end{equation*}
for any $z \in \mathbb{C}$, where the integral is interpreted as Bochner integral (see \cite{arendt2001vector}). It is also shown that there exists a real number $z_0:={\rm abs}(f)$ (eventually ${\rm abs}(f)=+\infty$), called the abscissa of convergence of $f$ such that $\cL_{t \to z}[f(t)](z)$ is convergent for any $z \in \mathbb{C}$ such that $\Re(z)>z_0$ and is divergent for any $z \in \mathbb{C}$ such that $\Re(z)<z_0$. We say that $f \in L^1_{\rm loc}([0,+\infty),X)$ is Laplace transformable if ${\rm abs}(f)<+\infty$.\\
In \cite{meerschaert2008triangular} the following relation has been shown:
\begin{equation}\label{Laptransdens}
\cL_{t \to z}[f_\Phi(s;t)](z)=\frac{\Phi(z)}{z}e^{-s\Phi(z)}, \ s>0, \ z>0.
\end{equation}
Moreover, it is easy to see that
\begin{equation}\label{laptranstail}
\cL_{t \to z}[\bar{\nu}_\Phi(t)](z)=\frac{\Phi(z)}{z}, \ z>0.
\end{equation}
In the following, we will use a particular class of Bernstein functions.
\begin{defn}
	Given a Bernstein function $\Phi \in \BF$ we call conjugate of $\Phi$ the function
	\begin{equation*}
	\Phi^\star(\lambda)=\frac{\lambda}{\Phi(\lambda)}.
	\end{equation*}
	We say $\Phi$ is a special Bernstein function if $\Phi^\star$ is a Bernstein function and the cone (see \cite[Proposition $11.20$]{schilling2012bernstein}) of special Bernstein function will be denoted by $\SBF$.
\end{defn}
Concerning special Bernstein functions, the associated potential measure is \textit{almost} absolutely continuous except at most for a jump in $0$, as stated in the following theorem (see \cite[Theorem $11.3$]{schilling2012bernstein}).
	\begin{thm}
		Let $\Phi \in \SBF$. Then there exists a non-negative and non-increasing function $u_\Phi(t)$ such that $\int_0^1 u_\Phi(t)dt<+\infty$ and
		\begin{equation*}
		U_\Phi(dt)=c_\Phi\delta_0(dt)+u_\Phi(t)dt
		\end{equation*}
		where $\delta_0(dt)$ is Dirac's $\delta$ measure centered in $0$ and
		\begin{equation*}
		c_\Phi=\begin{cases} 0 & b_\Phi>0 \\
		\frac{1}{a_\Phi+\nu_\Phi(0,+\infty)} & b_\Phi=0.
		\end{cases}
		\end{equation*}
		In particular, if $a_\Phi=b_\Phi=0$ and $\nu_\Phi(0,+\infty)=+\infty$, then $U_\Phi$ is absolutely continuous with density given by $u_\Phi$, called the potential density of $\Phi$.
	\end{thm} 
	We consider the following Assumption except where otherwise specified.
	\begin{assp}\label{assp}
		$\Phi \in \SBF$ and it holds $a_\Phi=b_\Phi=0$ and $\nu_\Phi(0,+\infty)=+\infty$. Moreover, there exist $t_0>0$, $C>0$ and $\beta \in (0,1)$ such that
		\begin{equation*}
		u_\Phi(t)\le Ct^{\beta-1}, \ t \in (0,t_0).
		\end{equation*} 
	\end{assp}
	\begin{rmk}\label{rmk:contU}
		The previous Assumption guarantees that for any $T>0$ there exists a constant $C$ such that
		\begin{equation*}
		U_\Phi(t)\le C t^\beta, \ t \in [0,T].
		\end{equation*}
	\end{rmk}

\section{Generalized Caputo derivatives and generalized fractional Ordinary Differential Equations}\label{Sec3}
Now we can introduce the main arguments of this work, i.e. the following generalization of Caputo derivatives.
\begin{defn}\label{def:genCap}
	Set a Banach space $(X,|\cdot|)$. For any function $f:[0,+\infty) \to X$ we define the generalized Caputo derivative of $f$ induced by $\Phi \in \BF$ as
	\begin{equation*}
	\partial_t^\Phi f(t)=\der{}{t}\int_0^t \overline{\nu}_\Phi(t-s)(f(s)-f(0))ds,
	\end{equation*}
	where the integral is intended as a Bochner integral, provided the involved quantities exist. In particular, if $f$ is absolutely continuous, then
	\begin{equation*}
	\partial_t^\Phi f(t)=\int_0^t \overline{\nu}_\Phi(t-s)\der{f}{t}(s)ds.
	\end{equation*}
\end{defn}
As we stated before, such operator are generalizations of the well-known Caputo fractional derivatives, that are achieved in the case $\Phi(\lambda)=\lambda^\alpha$. They were introduced in \cite{kochubei2011general} in the case of complete Bernstein functions and then in \cite{toaldo2015convolution} for general Bernstein functions.\\
Concerning the Laplace transform of $\partial_t^\Phi$, let us observe that if $\partial_t^\Phi f(t)$ is well-defined, then $\overline{\nu}_\Phi(\cdot) \ast (f(\cdot)-f(0))$ is absolutely continuous (where $\ast$ is the convolution product). Thus we can use \cite[Corollary $1.6.6$]{arendt2001vector} to guarantee that if $\cL_{t \to z}[\partial_t^\Phi f(t)](z)$ and $\cL_{t \to z}[f(t)](z)$ exist, then 
\begin{equation}\label{LaptransCap}
\cL_{t \to Z}[\partial_t^\Phi f(t)](\lambda)=\Phi(Z)\cL_{t \to z}[f(t)](\lambda)-\frac{\Phi(z)}{z}f(0).
\end{equation}
Concerning the inversion of generalized Caputo derivatives, we specifically need the assumption $\Phi \in \SBF$ as stated in \cite{meerschaert2019relaxation}.
\begin{defn}\label{def:genint}
	Set a Banach space $X$. For any function $f:[0,+\infty) \to X$, the generalized fractional integral of $f$ induced by $\Phi \in \SBF$ is defined as the operator
	\begin{equation*}
	\cI^\Phi_t f(t)=\int_0^{t}u_\Phi(t-\tau)f(\tau)d\tau
	\end{equation*}
	where $u_\Phi$ is the potential density of $\Phi$ and the integral is interpreted as Bochner integral, provided all the involved quantities exist.
\end{defn}
As shown in \cite{meerschaert2019relaxation}, if $f$ is a function such that $\partial_t^\Phi f$ is well defined, then
\begin{equation}\label{invid}
\cI^\Phi_t \partial_t^\Phi f(t)=\partial_t^\Phi \cI^\Phi_t f(t)=f(t)-f(0)
\end{equation}
thus we can see the operator $\cI^\Phi_t$ as the inverse of the generalized Caputo derivative $\partial_t^\Phi$. In the case $\Phi(\lambda)=\lambda^\alpha$ we obtain the usual fractional integral.\\
From now on, we fix a Banach space $(X,|\cdot|)$. In this work we want to study Cauchy problems of the form
\begin{equation}\label{Cprob}
\begin{cases}
\partial_t^\Phi f(t)=F(t,f(t)) & \mbox{a.e. }t \in [0,T]\\
f(0)=f_0.
\end{cases}
\end{equation} 
In particular we want to show an existence and uniqueness result via Picard iterations and contraction theorem. To do this, as usual, we first need to recast the aforementioned Cauchy problem as an integral equation.
\begin{lem}
	Let $F:[0,T] \times X \to X$ and $f_0 \in X$. The function $f:[0,T] \to X$ is a solution of the following generalized fractional Cauchy problem
	\begin{equation}\label{CP}
	\begin{cases}
	\partial_t^\Phi f(t)=F(t,f(t)) & \mbox{a.e. }t \in [0,T]\\
	f(0)=f_0.
	\end{cases}
	\end{equation}
	if and only if it is a solution of the integral equation
	\begin{equation}\label{IE}
	f(t)=f_0+\cI_t^\Phi F(t,f(t)), \qquad t \in [0,T],
	\end{equation}
	provided that $F(\cdot,f(\cdot)) \in L^1([0,T],X)$.
\end{lem} 
\begin{proof}
	Let us first suppose that $f$ is solution of \eqref{CP}. Then, applying $\cI^\Phi_t$ to both sides of the first equation, substituting $f(0)=f_0$ and by relation \eqref{invid}, we get Equation \eqref{IE}.\\
	Vice versa, if we suppose $f$ is solution of \eqref{IE}, the $f(t)$ is continuous since the right-hand side of \eqref{IE} is continuous. Moreover, we can apply the operator $\partial_t^\Phi$ on both sides of such equation (since the right-hand side admits a generalized Caputo derivative), recalling that $\partial_t^\Phi f_0=0$, to achieve the first equation of \eqref{CP}. The initial condition follows from the fact that $\cI_0^\Phi f(t)=f_0$.
\end{proof}
Now that we have such Lemma, we are ready to show the local existence, uniqueness and regularity Theorem. Before giving its statement, let us set some notation. For any time interval $J=[0,T]$ and any $K \subseteq X$ we denote by $C(J,K)$ the space of continuous functions $f:J \to K$. For any $\beta \in (0,1)$ we denote by $C^\beta(J,K)$ the space of H\"older-continuous functions $f:J \to K$ of order $\beta$, i.e. such that for some $L>0$
\begin{equation*}
|f(t)-f(s)|\le L|t-s|^\beta, \ t,s \in J.
\end{equation*}
Now we can give the statement of the main Theorem of this section.
\begin{thm}\label{ex}
	Let $\Phi\in \SBF$ satisfy Assumption \ref{assp}, set $J=[0,T]$ and $F:J \times X \to X$. Suppose the following conditions hold:
	\begin{itemize}
		\item[$A1$] For any ball $B_R=\{x \in X: \ |x|<R\}$ in $X$ there exists a constant $C_R>0$ such that $|F(t,x)|\le C_R$ for almost any $t \in J$ and any $x \in B_R$;
		\item[$A2$] For any ball $B_R$ in $X$ there exists a constant $L_R>0$ such that $|F(t,x)-F(t,z)|\le L_R|x-z|$ for almost any $t \in J$ and any $x,z \in B_R$.
	\end{itemize}
	Then for any initial datum $f_0 \in X$ there exists $T'>0$ such that equation \eqref{Cprob} admits a unique solution $f \in C^{\beta}(J',B_R(f_0))$ where $J'=[0,T']$ and $B_R(f_0)=\{x \in X: \ |x-f_0|<R\}$.
\end{thm}
The proof of the Theorem will be articulated in the following two Subsections:
\begin{itemize}
	\item We first need to introduce what will be our Picard iterative operator $A$, taking in consideration the structure of equation \eqref{IE} and showing that for any $f \in C(J,B_R(f_0))$, $Af \in C^{\beta}(J,B_R(f_0))$;
	\item Then we just need to show that the operator is actually a contraction when we choose the right norm on $C(J',B_R(f_0))$.
\end{itemize}
\subsection{Definition of the Picard operator}
Let us define the following Bielecki-type norm on $C(J,X)$ for any $\tau\ge 0$ as
\begin{equation*}
f \in C(J,X) \mapsto \Norm{f}{\tau}=\max_{t \in J}|f(t)|e^{-\tau t},
\end{equation*}
where for $\tau=0$ we have the usual supremum norm. Fix $f_0 \in X$ and $R>0$ and consider
\begin{equation*}
A:(C(J,B_R(f_0)), \Norm{\cdot}{\tau}) \to (C(J,X), \Norm{\cdot}{\tau})
\end{equation*}
defined as
\begin{equation*}
Af(t)=f_0+\cI_t^\Phi F(t,f(t));
\end{equation*}
$A$ will be our Picard operator, observing that any fixed point of $A$ is solution of \eqref{IE}. 
Let us first show that $A$ is well defined.
\begin{lem}\label{lemreg}
	The operator $A$ is well defined and its range is contained in $C^{\beta}(J,X)$.
\end{lem} 
\begin{proof}
Fix $\delta>0$, $\varepsilon \in (0,\beta)$ and $f \in C(J,B_R(f_0))$. First of all, let us recall that if $x \in B_R(f_0)$, then $|x|\le |x-f_0|+|f_0|<R+|f_0|$. Defining $\widetilde{R}=R+|f_0|$, we have $f(t)\in B_{\widetilde{R}}$ for any $t \in J$. By definition of $A$ we have
\begin{align}
\label{estA}
\begin{split}
|Af(t+\delta)-Af(t)|&= \left|\int_0^{t}u_\Phi(t-s)F(s,f(s))ds-\int_0^{t+\delta}u_\Phi(t+\delta-s)F(s,f(s))ds\right|\\
&\le \int_0^{t}|u_\Phi(t-s)-u_\Phi(t+\delta-s)||F(s,f(s))|ds\\
&+\int_t^{t+\delta}u_\Phi(t+\delta -s)|F(s,f(s))|ds\\
&=I_1(t)+I_2(t).
\end{split}
\end{align}
Let us first consider $I_2(t)$. Since $f(t) \in B_{\widetilde{R}}$ for any $t \in [0,T]$, by hypothesis $A1$ we have
\begin{equation*}
I_2(t)\le  C_{\widetilde{R}}\int_t^{t+\delta}u_\Phi(t+\delta-s)ds=C_{\widetilde{R}}U(\delta)\le C\delta^{\beta}.
\end{equation*}
Concerning $I_1$ we have
\begin{equation*}
I_1(t)\le C_{\widetilde{R}}(U(t+\delta)-U(t)+U(\delta))\le 2C_{\widetilde{R}}U(\delta)\le C\delta^\beta,
\end{equation*}
where we used the subadditivity of the function $U$. Hence we have
\begin{equation*}
|Af(t+\delta)-Af(t)|\le C\delta^{\beta},
\end{equation*}
concluding the proof.
\end{proof} 
Actually, we need the range to be contained in $C(J,B_R(f_0))$, thus we will need the following technical Lemma.
\begin{lem}\label{lemT1}
	There exists $T'(R) \le T$ such that for any $f \in C(J,B_R(f_0))$ it holds $Af_{|J'} \in C(J',B_R(f_0))$.
\end{lem}
\begin{proof}
Arguing as before we have
\begin{equation*}
|Af(t)-f_0|\le\int_0^t u(t-s)|F(s,f(s))|ds\le C_{\widetilde{R}} U(t).
\end{equation*}
Now let us choose $T'>0$ such that $C_{\widetilde{R}}U(T')<R$ to obtain
\begin{equation*}
|Af(t)-f_0|\le C_{\widetilde{R}} U(t)\le  C_{\widetilde{R}} U(T')<R,
\end{equation*}
concluding the proof.
\end{proof}
\begin{rmk}
	The previous Lemma is the only one in which a condition on the time horizon $T'$ is imposed. This observation will come useful in what follows.
\end{rmk}
In particular, now we can consider
\begin{equation*}
A:(C(J',B_R(f_0)), \Norm{\cdot}{\tau}) \to (C(J',B_R(f_0)), \Norm{\cdot}{\tau}),
\end{equation*}
recalling that the range of $A$ is contained in $C^\beta(J',B_R(f_0))$.
\subsection{Contraction property with respect to the Bielecki norm}
Now let us show that $A$ is a contraction for some choice of $\tau>0$. Let us first recall that this norm is equivalent to the one with $\tau=0$. Indeed we have
\begin{equation*}
|f(t)|e^{-\tau T'}\le |f(t)|e^{-\tau t}\le |f(t)|
\end{equation*}
that, taking the maximum for $t \in J'$, gives the equivalence. Let us also recall that $(C(J',B_R(f_0)),\Norm{\cdot}{\tau})$ is a Banach space for any $\tau>0$.\\
Now let us show the following Proposition.
\begin{prop}
There exists $\tau$ such that $A$ is a contraction on $(C(J',B_R(f_0)), \Norm{\cdot}{\tau})$.
\end{prop}
\begin{proof}
	Let us consider $f,g \in C(J',B_R(f_0))$ and consider $\widetilde{R}$ as defined before. We have
	\begin{align*}
	|Af(t)-Ag(t)|&\le \int_0^t u(t-s)|F(s,f(s))-F(s,g(s))|ds\\
	&\le L_{\widetilde{R}}\int_0^t u(t-s)|f(s)-g(s)|ds\\
	&= L_{\widetilde{R}}\int_0^t u(t-s)|f(s)-g(s)|e^{-\tau s}e^{\tau s}ds\\
	&\le L_{\widetilde{R}}\Norm{f-g}{\tau}\int_0^t u(t-s)e^{\tau s}ds\\
	&\le CL_{\widetilde{R}}\Norm{f-g}{\tau}\int_0^t (t-s)^{\beta-1}e^{\tau s}ds.
	\end{align*}
	Consider $p \in \left(1,\frac{1}{1-\beta}\right)$ (for instance $p=\frac{2-\beta}{2(1-\beta)}$) and use H\"older inequality to achieve
	\begin{align*}
	|Af(t)-Ag(t)|&\le  C L_{\widetilde{R}}\Norm{f-g}{\tau}\left(\int_0^t (t-s)^{p(\beta-1)}ds\right)^{\frac{1}{p}}\left(\int_0^t e^{p'\tau s}ds\right)^{\frac{1}{p'}}\\
	&=CL_{\widetilde{R}}\Norm{f-g}{\tau}\frac{1}{(p(\beta-1)+1)^{\frac{1}{p}}} T'^{\frac{p(\beta-1)+1}{p}}\left(\frac{1}{p'\tau}\right)^{\frac{1}{p'}}e^{\tau t}.
	\end{align*}
	Multiplying both sides of the previous equation by $e^{-\tau t}$ and then taking the maximum, we achieve
	\begin{align*}
	\Norm{Af-Ag}{\tau}&\le CL_{\widetilde{R}}\Norm{f-g}{\tau}\frac{1}{(p(\beta-1)+1)^{\frac{1}{p}}} T'^{\frac{p(\beta-1)+1}{p}}\left(\frac{1}{p'\tau}\right)^{\frac{1}{p'}}.
	\end{align*}
	Finally, choose $\tau>0$ big enough such that
	$$CL_{\widetilde{R}}\frac{1}{(p(\beta-1)+1)^{\frac{1}{p}}} T'^{\frac{p(\beta-1)+1}{p}}\left(\frac{1}{p'\tau}\right)^{\frac{1}{p'}}<1,$$
	concluding the proof.
\end{proof}

Now let us end the proof of Theorem \ref{ex}. Since $A$ is a contraction over \linebreak $(C(J',B_R(f_0)),\Norm{\cdot}{\tau})$, we know by contraction theorem (see \cite{khamsi2011introduction}) that there exists a unique fixed point $f \in C(J',B_R(f_0))$ of $A$. In particular, $f$ is the solution we are searching for. Moreover, since the range of $A$ is contained in $C^\beta(J',B_R(f_0))$ and $f=Af$ belongs to the range of $A$, we know that $f \in C^\beta(J',B_R(f_0))$, concluding the proof.
\qed
\subsection{The affine autonomous case}
As main example, let us show an easy case as the affine autonomous case.
\begin{cor}\label{exandun}
	Let $F:X \to X$ be a bounded linear operator and let $\xi \in X$. Suppose $\Phi \in \SBF$ satisfies Assumption \ref{assp}. Then there exists $T$ depending only on $\Norm{F}{L(X,X)}:=\sup_{|x|=1}|Fx|$ such that the problem
\begin{equation}\label{lCP}
\begin{cases}
\partial_t^\Phi f(t)=\xi+F f(t) & t \in [0,T)=:J\\
f(0)=f_0.
\end{cases}
\end{equation}
admits a unique solution $f \in C^\beta(J,X)$.
\end{cor}
\begin{proof}
	Let us first observe that hypotheses $A1$ and $A2$ are satisfied for any $T>0$. In particular, for any $T>0$ and $R>0$, we have $L_R=\Norm{F}{L(X,X)}$ and we can consider
	\begin{equation*}
	C_R=|\xi|+\Norm{F}{L(X,X)}R.
	\end{equation*}
	Now we want to characterize $T'$ as defined in \ref{lemT1}. To do this, let us consider the left-continuous inverse of $U$
	\begin{equation*}
	U^{\leftarrow}(u)=\min\{x \ge 0 :U(x)\ge u\}, \qquad \forall u>0
	\end{equation*}
	and let us observe that we can define $T(R):=T'=U^{\leftarrow}\left(\frac{R}{|\xi|+\Norm{F}{L(X,X)}R+\Norm{F}{L(X,X)}|f_0|}\right)$. Thus we have, for any $R>0$, a unique solution in $C^\beta(J(R),B_R(f_0))$ where $J(R)=[0,T(R)]$. Sending $R \to +\infty$ we conclude the proof.
\end{proof}
\begin{rmk}
	Let us observe that if $\Phi \in \SBF$ satisfies Assumption \ref{assp}, then, considering any $\alpha \in (1-\beta,1)$, we have that $(t-s)^\alpha u(t-s)$ is integrable. In particular this means that the affine autonomous case can be treated as a Fredholm equation of the second kind and thus the solution can be expressed by Neumann series (see \cite{fredholm1903classe}). 
\end{rmk}
In the specific case of affine autonomous problems, we can show that if we have a global solution, then it is unique.
\begin{cor}\label{corollaryexist}
	Under the hypotheses of Corollary \ref{exandun}, suppose $\nu_\Phi$ is absolutely continuous with respect to Lebesgue measure. If there exists a solution \linebreak $f \in C([0,+\infty),X)$, then this solution is unique.
\end{cor}
\begin{proof}
	Let us suppose there are two solutions $f_1(t)$ and $f_2(t)$ of \eqref{lCP} defined on $[0,+\infty)$. By Corollary \ref{exandun} we already know there exists $T'>0$ such that $f_1(t)=f_2(t)$ for any $t \in [0,T']$ (just consider any $T'<T$). Let us then consider $T_*=\sup\{t>0: f_1(t)=f_2(t)\}$. We want to show that $T_*=+\infty$. Let us argue by contradiction: suppose that $T_*<+\infty$. Since $f_1$ and $f_2$ are continuous, we know that $f_1(T_*)=f_2(T_*)=f_*$. Let us consider $t>T_*$ and observe that
	\begin{align*}
	\partial_t^\Phi f_1(t)&=\der{}{t}\int_0^t \overline{\nu}_\Phi(t-s)(f_1(s)-f_0)ds\\
	&=\der{}{t}\left(\int_0^{T_*} \overline{\nu}_\Phi(t-s)(f_1(s)-f_0)ds+\int_{T_*}^t \overline{\nu}_\Phi(t-s)(f_1(s)-f_*)ds\right)\\&\qquad +(f_*-f_0)\overline{\nu}_\Phi(t-T_*).
	\end{align*}
	Let us denote by $\nu_\Phi(t)$ the density of the L\'evy measure $\nu_\Phi$. Then we have $\der{}{t}\bar{\nu}_\Phi(t-s)=-\nu_\Phi(t-s)$ and, since $f_1$ is continuous,
	\begin{equation*}
	\int_0^{T_*}\nu_\Phi(t-s)|f_1(s)-f_0|ds\le C\int_0^{T_*}\nu_\Phi(t-s)ds=C\nu_\Phi(t-T^*,t),
	\end{equation*}
	thus we have
	\begin{align*}
	\partial_t^\Phi f_1(t)&=-\int_0^{T_*} \nu_\Phi(t-s)(f_1(s)-f_0)ds+\der{}{t}\int_{T_*}^t \overline{\nu}_\Phi(t-s)(f_1(s)-f_*)ds\\&\qquad +(f_*-f_0)\overline{\nu}_\Phi(t-T_*).
	\end{align*}
	Recalling the definition of $f_1$, we have
	\begin{align*}
	\der{}{t}\int_{T_*}^t \overline{\nu}_\Phi(t-s)&(f_1(s)-f_*)ds=Ff_1(t)+\xi\\&+\int_0^{T_*} \nu_\Phi(t-s)(f_1(s)-f_0)ds -(f_*-f_0)\overline{\nu}_\Phi(t-T_*).
	\end{align*}
	The same relation holds for $f_2$.
	Now define $g=f_1-f_2$ and observe that
	\begin{equation*}
	\der{}{t}\int_{T_*}^t \overline{\nu}_\Phi(t-s)g(s)ds=Fg(t),
	\end{equation*}
	since $g \equiv 0$ on $[0,T_*]$. Now let us consider the change of variables $w=s-T_*$. To obtain
	\begin{equation*}
	\der{}{t}\int_{0}^{t-T_*} \overline{\nu}_\Phi(t-T_*-w)g(T_*+w)dw=Fg(T_*+t-T_*).
	\end{equation*}
	Defining the function $h(t)=g(T_*+t)$ for $t>0$ we finally have
	\begin{equation*}
	\partial_t^\Phi h(t)=Fh(t).
	\end{equation*}
	Recalling that $h(0)=0$, we have that $h$ satisfies the Cauchy problem
	\begin{equation*}
	\begin{cases}
	\partial_t^\Phi h(t)=Fh(t) & t>0 \\
	h(0)=0.
	\end{cases}
	\end{equation*}
	However, by Theorem \ref{ex}, we know that the previous problem admits a unique solution in $[0,T']$ and $h(t)\equiv 0$ is a solution. Hence in $[0,T']$ we have $h(t)\equiv 0$, which is absurd since this implies that $f_1(t)=f_2(t)$ for $t \in [T_*,T_*+T']$. Thus $T_*=+\infty$ and we conclude the proof.
\end{proof}
\section{Eigenvalue problems and $\Phi$-exponential functions}\label{Sec4}
In this section we will focus on the eigenvalue problem
\begin{equation}\label{reqeig}
\begin{cases}
\partial_t^\Phi f(t)=\lambda f(t) & t>0\\
f(0)=f_0,
\end{cases}
\end{equation}
where $\lambda \in \R$, $f:[0,+\infty) \to \R$ and $f_0 \in \R$. Let us stress that problem \eqref{reqeig} for $\lambda<0$ has been studied in \cite{kochubei2011general} when $\Phi$ is a complete Bernstein function and in \cite{meerschaert2019relaxation} when $\Phi$ is a special Bernstein function. On the other hand, the case $\lambda>0$ has been studied (up to our knowledge) only in \cite{kochubei2019growth} as $\Phi$ is a complete Bernstein function. Here we want to show that problem \eqref{reqeig} admits a solution as $\Phi$ is any Bernstein function with $a_\Phi=b_\Phi=0$ and $\nu_\Phi(0,+\infty)=+\infty$ and for any $\lambda \in \R$.\\ 
To do this, we need the following Lemma.
\begin{lem}
	Let $\Phi \in \BF$ be a Bernstein function such that $a_\Phi=b_\Phi=0$ and $\nu_\Phi(0,+\infty)=+\infty$. Then the function
	\begin{equation*}
	\fe_\Phi(t;\lambda)=\E[e^{\lambda L_\Phi(t)}]
	\end{equation*}
	is well defined for any $t>0$ and $\lambda \in \R$.
\end{lem}
\begin{proof}
	Let us observe that the statement is obvious as $\lambda \le 0$, so we only have to focus on the case $\lambda>0$. Let us observe that for any $x>0$ it holds
	\begin{equation*}
	\bP(L_\Phi(t)>s)\le \bP(\sigma_\Phi(s)\le t)=\bP(e^{-x\sigma_\Phi(s)}\ge e^{-xt}).
	\end{equation*}
	By Markov's inequality, recalling equation \eqref{Lapexp}, we achieve
	\begin{equation}\label{control}
	\bP(L_\Phi(t)>s)\le e^{xt-s\Phi(x)}.
	\end{equation}
	Now let us recall observe that, obviously, $e^{\lambda L_\Phi(t)}$ is a non-negative random variable, hence
	\begin{equation*}
	\E[e^{\lambda L_\Phi(t)}]=\int_0^{+\infty}\bP(e^{\lambda L_\Phi(t)}>s)ds=\int_0^{1}\bP(e^{\lambda L_\Phi(t)}>s)ds+\int_1^{+\infty}\bP(e^{\lambda L_\Phi(t)}>s)ds.
	\end{equation*}
	As $\int_0^{1}\bP(e^{\lambda L_\Phi(t)}>s)ds \le 1$, we only have to consider the second integral. We have, by equation \eqref{control}
	\begin{align*}
	\int_1^{+\infty}\bP(e^{\lambda L_\Phi(t)}>s)ds&=\int_1^{+\infty}\bP\left(L_\Phi(t)>\frac{\log(s)}{\lambda}\right)ds\\
	&\le e^{xt}\int_1^{+\infty}e^{-\frac{\Phi(x)}{\lambda}\log(s)}ds\\
	&=e^{xt}\int_1^{+\infty}s^{-\frac{\Phi(x)}{\lambda}}ds.
	\end{align*}
	However, $\Phi(x)$ is unbounded (being $\nu_\Phi(0,+\infty)=+\infty$), thus we can consider $x>0$ big enough to have $\frac{\Phi(x)}{\lambda}>1$, concluding the proof.
\end{proof}
\begin{rmk}\label{MLrmk}
	Let us recall that if $\Phi(\lambda)=\lambda^\alpha$ for $\alpha \in (0,1)$, it has been shown in \cite{bingham1971limit} that
	\begin{equation*}
	\fe_\Phi(t;\lambda)=E_\alpha(\lambda t^\alpha),
	\end{equation*}
	where $E_\alpha$ is the one-parameter Mittag-Leffler function
	\begin{equation*}
	E_\alpha(z)=\sum_{k=0}^{+\infty}\frac{z^k}{\Gamma(\alpha k+1)}, \ z \in \mathbb{C}.
	\end{equation*}
\end{rmk}
Arguing via Laplace transforms, we can show that the function $f_0\fe_\Phi(t;\lambda)$ is the solution of the eigenvalue problem \eqref{reqeig}.
\begin{prop}\label{prop:eigenfunc}
	Let $\Phi \in \BF$ be a Bernstein function with $a_\Phi=b_\Phi=0$ and $\nu_\Phi(0,+\infty)$. Then, for any $\lambda \in \R$, the function $f(t)=f_0\fe_\Phi(t;\lambda)$ is the unique solution of the Cauchy problem \eqref{reqeig}.
\end{prop}
\begin{proof}
Let us first show that $f(t)=f_0\fe_\Phi(t;\lambda)$ is a solution of the Cauchy Problem \eqref{reqeig}. Let us observe that $\fe_\Phi(t;0)=1$, thus, if $\lambda=0$, $f(t)=f_0$ is a solution of the aforementioned problem.\\
Now let us consider $\lambda\not =0$. Since $\fe_\Phi(0;\lambda)=1$, we have, by definition, $f(0)=f_0$. Using Equation \eqref{Laptransdens} and denoting $\bar{f}(z)=\cL_{t \to z}[f(t)](z)$, we have, by definition,
\begin{align*}
\bar{f}(z)&=f_0\cL_{t \to z}\left[\int_0^{+\infty}e^{\lambda s}f_L(s;t)ds\right](z)\\
&=f_0\frac{\Phi(z)}{z(\Phi(z)-\lambda)}
\end{align*}
for any real $z$ such that $\Phi(z)>\lambda$. In particular we obtain that ${\rm abs}(f)=\Phi^{-1}(\lambda)$. After some simple algebraic manipulation we have
\begin{equation}\label{Laptranseq}
\frac{\Phi(z)}{z}\left(\bar{f}(z)-\frac{f_0}{z}\right)=\frac{\lambda}{z} \bar{f}(z).
\end{equation}
Let us recall that $L_\Phi(t)$ is non-negative and increasing, thus also $\fe_\Phi(t;\lambda)$ for any $\lambda>0$, while it is controlled by $1$ if $\lambda<0$. Hence $f(t)$ belongs to $L^\infty_{\rm loc}(0,+\infty)\subset L^1_{\rm loc}(0,+\infty)$. This means that the function $F(t)=\int_0^tf(s)ds$ is well defined and absolutely continuous. Being $\Phi^{-1}(\lambda)>0$, by \cite[Corollary $1.6.5$]{arendt2001vector}, we have ${\rm abs}(F)\le \Phi^{-1}(\lambda)$ and $\cL_{t \to z}[F(t)](z)=\frac{\bar{f}(z)}{z}$ as $z>\Phi^{-1}(\lambda)$. On the other hand, $\cL_{t \to z}[f(t)-f_0]=\bar{f}(z)-\frac{f_0}{z}$ as $z>\Phi^{-1}(\lambda)$, thus, by Equation \eqref{laptranstail} and \cite[Proposition $1.6.4$]{arendt2001vector}, we have, starting from Equation \eqref{Laptranseq},
\begin{equation*}
\cL_{t \to z}[(\bar{\nu}_\Phi(\cdot)\ast (f(\cdot)-f_0))(t)](z)=\cL_{t \to z}[\lambda F(t)](z).
\end{equation*}
By injectivity of the Laplace transform we finally have
\begin{equation*}
\int_0^t\bar{\nu}_\Phi(t-s)(f(s)-f_0)ds=\lambda \int_0^tf(s)ds.
\end{equation*}
Now let us observe that the right-hand side of the previous equation is absolutely continuous (since $f \in L^1_{\rm loc}(0,+\infty)$), thus we can differentiate both sides for almost any $t>0$, leading to the first equation of \eqref{reqeig}. Finally, uniqueness follows from injectivity of the Laplace transform.
\end{proof}
Since the function $\fe_\Phi(t;\lambda)$ \textit{plays the role of the exponential} with respect to $\partial_t^\Phi$, we call them $\Phi$-exponential functions.\\
Now let us focus again on the case $\Phi \in \SBF$ satisfying Assumption \ref{assp}. In this case, let us define the following sequence of functions on $(0,+\infty)$:
\begin{equation*}
\begin{cases}
u_0^*(t)=1;\\
u_1^*(t)=U_\Phi(t);\\
u_{k+1}^*(t)=\int_0^t u_\Phi(t-s)u_k^*(s)ds & k \ge 1.
\end{cases}
\end{equation*}
We want to express the $\Phi$-exponential function in terms of series of functions by using the sequence $(u_k^*(t))_{k \ge 1}$. To do this, let us first show the following technical Lemma
\begin{lem}\label{lem:convsum}
	Let $\Phi \in \SBF$ satisfying Assumption \ref{assp}. Then for any fixed $\lambda>0$ the series
	\begin{equation*}
	\sum_{k=0}^{+\infty}\lambda^ku_k^*(t)
	\end{equation*}
	is normally convergent in $[0,T]$ for any $T>0$.
\end{lem}
\begin{proof}
	Let us fix $T>0$ and recall, by Remark \ref{rmk:contU}, that there exists some constant $C_1>0$ (depending on $T$) such that $U(t)\le C_1t^{\beta}$ for any $t \in [0,T]$. On the other hand, by Assumption \ref{assp}, we have $u(t)\le C_2t^{\beta-1}$ for any $t \in [0,T]$ for some constant $C_2>0$ depending on $T$. We want to show that, for $k \ge 1$,
	\begin{equation}\label{indassp1}
	u_k^*(t)\le C_1C^{k-1}_2\frac{\beta}{\Gamma(k\beta+1)}(\Gamma(\beta)t^\beta)^{k}.
	\end{equation}
	This obviously holds as $k=1$. Let us suppose inequality \eqref{indassp1} holds for some $k \ge 1$. Then we have
	\begin{align*}
	u_{k+1}^*(t)&\le C_1C^{k}_2\frac{\beta\Gamma(\beta)^k}{\Gamma(k\beta+1)}\int_0^{t}(t-s)^{\beta-1}s^{k\beta}ds\\
	&=C_1C^{k}_2\frac{\beta\Gamma(\beta)^k}{\Gamma(k\beta+1)}t^{(k+1)\beta}\int_0^{1}(1-w)^{\beta-1}w^{k\beta}dw\\
	&=C_1C^{k}_2\frac{\beta\Gamma(\beta)^{k+1}}{\Gamma((k+1)\beta+1)}t^{(k+1)\beta},
	\end{align*}
	where we used the change of variables $w=\frac{s}{t}$. Hence we have
	\begin{equation*}
	\sum_{k=0}^{+\infty}\lambda^ku_k^*(t)\le 1+\frac{C_1\beta}{C_2}\sum_{k=1}^{+\infty}\frac{(\lambda C_2\Gamma(\beta) t^\beta)^k}{\Gamma(\beta k+1)}=1+\frac{C_1\beta}{C_2}(E_\beta(\lambda C_2 \Gamma(\beta)t^\beta)-1),
	\end{equation*}
	where $E_\beta$ is the one-parameter Mittag-Leffler function defined in Remark \ref{MLrmk}.
\end{proof}
Then we are ready to show the following Theorem.
\begin{thm}\label{thm:series}
	Let $\Phi \in \SBF$ satisfying Assumption \ref{assp}. Then for any $\lambda \in \R$ it holds
	\begin{equation*}
	\fe_\Phi(t;\lambda)=\sum_{k=0}^{+\infty}\lambda^ku_k^*(t).
	\end{equation*}
\end{thm}
\begin{proof}
	For $\lambda=0$ we have $\fe_\Phi(t,0)\equiv 1$, thus let us consider $\lambda \not = 0$. Let us recall that, in the proof of Proposition \ref{prop:eigenfunc}, we have shown that for $z>\Phi^{-1}(\lambda)$ it holds
	\begin{equation*}
	\cL_{t \to z}[\fe_\Phi(t;\lambda)]=\frac{1}{z\left(1-\frac{\lambda}{\Phi(z)}\right)}=\sum_{k=0}^{+\infty}\frac{\lambda^k}{z\Phi^k(z)}.
	\end{equation*}
	Now let us show that for any $k \ge 0$ it holds 
	\begin{equation}\label{formLapt}
	\cL_{t \to z}[u_k^*(t)]=\frac{1}{z\Phi^k(z)}
	\end{equation}
	This is obviously true for $k=0$. Thus, let us suppose that Equation \eqref{formLapt} holds for some $k \ge 0$. Let us recall, from \cite{veillette2010using}, that 
	\begin{equation*}
	\cL_{t \to z}[U_\Phi(t)](z)=\frac{1}{z\Phi(z)} \  \mbox{ and } \  \cL_{t \to z}[u_\Phi(t)](z)=\frac{1}{\Phi(z)}
	\end{equation*}
	as $z>0$. By \cite[Proposition $1.6.4$]{arendt2001vector} we have
	\begin{equation*}
	\cL_{t \to z}[u_{k+1}^*(t)](z)=\cL_{t \to z}[u_{k}^*(t)](z)\cL_{t \to z}[u_\Phi(t)](z)=\frac{1}{z\Phi^{k+1}(z)}.
	\end{equation*}
	Now let us consider $\Phi(z)>\max\{\lambda,0\}$ and observe that, by monotone convergence theorem if $\lambda>0$ and dominated convergence theorem if $\lambda<0$, it holds
	\begin{equation*}
	\cL_{t \to z}\left[\sum_{k=0}^{+\infty}\lambda^k u_k^*(t)\right](z)=\sum_{k=0}^{+\infty}\frac{\lambda^k}{z\Phi^k(z)},
	\end{equation*}
	concluding the proof by injectivity of the Laplace transform.
\end{proof}
\begin{rmk}
	By comparing the Laplace transform, we have
	\begin{equation*}
	u_k^*(t)=\frac{\E[L_\Phi^k(t)]}{k!}, \ t \ge 0, \ k \ge 0.
	\end{equation*}
\end{rmk}
Now that we have some properties concerning the eigenfunctions of the generalized Caputo derivative, we could ask if they play the role of the exponential in a generalization of Gr\"onwall inequality. Thus, we now close this \textit{linear parenthesis} and move forward to show a generalization of the Gr\"onwall inequality, which will be then used to determine some properties of the non-local Cauchy problem \eqref{Cprob}, such as continuity with respect to initial data and other parameters.
\section{The generalized Gr\"onwall inequality}\label{Sec5}
As we stated before, now we focus on a Gr\"onwall-type inequality for the generalized fractional integral operator given in Definition \ref{def:genint}. In particular, we will follow the lines of \cite[Theorem $1$]{ye2007generalized}. Let us first state the Theorem.
\begin{thm}\label{thm:Gron}
	Let $x,a,g \in L^1(0,T)$ with $a,g \ge 0$ almost everywhere and $g$ non-decreasing. Let $\Phi \in \SBF$ satisfying Assumption \ref{assp}. Moreover, suppose that
	\begin{equation}\label{intineq}
	x(t)\le a(t)+g(t)\cI_t^\Phi x(t), \ t \in [0,T].
	\end{equation}
	Then:
	\begin{itemize}
		\item It holds
		\begin{equation*}
		x(t)\le \sum_{k=0}^{+\infty}B^ka(t)
		\end{equation*}
		where $B^0$ is the identity operator and $B$ is defined as
		\begin{equation*}
		Ba(t)=g(t)\int_0^tu(t-s)a(s)ds;
		\end{equation*}
		\item There exists $C>0$ such that
		\begin{equation}\label{est5}
		x(t)\le a(t)+C\Gamma(\beta+1)g(t)\int_0^t E'_{\beta}(C\Gamma(\beta)g(t)(t-s))(t-s)^{\beta-1}a(s)ds,
		\end{equation}
		where $E_{\beta}(t)$ is the Mittag-Leffler function defined in Remark \ref{MLrmk}.
		\item If $a$ is non-decreasing then
		\begin{equation*}
		x(t)\le a(t)\fe_\Phi(t,g(T)).
		\end{equation*}
\end{itemize}
\end{thm} 
Before proving the Theorem, let us give some technical Lemmas concerning the operator $B$. 
\subsection{The auxiliary operator $B$}
In this section we consider $\Phi \in \SBF$ satisfying Assumption \ref{assp}. For any function $f \in L^1([0,T])$ let us define, as stated in Theorem \ref{thm:Gron}, the following operator
\begin{equation*}
Bf(t)=g(t)\int_0^tu_\Phi(t-s)f(s)ds.
\end{equation*}
In this subsection we will prove some properties of the operator $B$ and of its powers. First of all, let us observe that, being $g$ and $u_\Phi$ non-negative, then for any $f_1,f_2 \in L^1([0,T])$ with $f_1(t)\le f_2(t)$ almost everywhere in $[0,T]$, it holds $Bf_1(t)\le Bf_2(t)$ for any $t \in [0,T]$. This simple property leads to an upper bound on $B^kf(t)$ as $f \in L^1(0,T)$.
\begin{lem}\label{PEBk}
Then there exists a constant $C$ such that for any $k \ge 1$ and any non-negative $f \in L^1([0,T])$ it holds
\begin{equation}\label{Bkcontrol}
B^kf(t)\le \frac{(C\Gamma(\beta)g(t))^k}{\Gamma(k\beta)}\int_0^t(t-s)^{k\beta-1}f(s)ds.
\end{equation}
\end{lem} 
\begin{proof}
	By Assumption \ref{assp} we know there exists a constant $C>0$ such that $u_\Phi(t)\le Ct^{\beta-1}$. Then, for $k=1$ we have
	\begin{equation*}
	Bf(t)=g(t)\int_0^t u_\Phi(t-s)f(s)ds\le Cg(t) \int_0^t (t-s)^{\beta-1}f(s)ds.
	\end{equation*}
	Now let us suppose Equation \eqref{Bkcontrol} holds for some $k \ge 1$. Then we have
	\begin{align*}
	B^{k+1}f(t)&=B(B^kf(t))(t)\le g(t)\int_0^tu_\Phi(t-s)\frac{(C\Gamma(\beta)g(s))^k}{\Gamma(k\beta)}\int_0^s(s-\tau)^{k\beta-1}f(\tau)d\tau ds\\
	&\le \frac{(C\Gamma(\beta)g(t))^{k+1}}{\Gamma(k\beta)\Gamma(\beta)}\int_0^t (t-s)^{\beta}\int_0^s(s-\tau)^{k\beta-1}f(\tau)d\tau ds\\
	&=\frac{(C\Gamma(\beta)g(t))^{k+1}}{\Gamma(k\beta)\Gamma(\beta)}\int_0^t f(\tau)(t-\tau)^{(k+1)\beta-1}\int_0^1(1-w)^{\beta-1}w^{k\beta-1}dw d\tau\\
	&=\frac{(C\Gamma(\beta)g(t))^{k+1}}{\Gamma((k+1)\beta)}\int_0^t f(\tau)(t-\tau)^{(k+1)\beta-1}d\tau,
	\end{align*}
	where we used the fact that $g$ is non-decreasing and the change of variables \linebreak $s=\tau+(t-\tau)w$.
\end{proof}
In the case $f(t)\equiv 1$, we can control $B^k1(t)$ with the functions $u^*_k(t)$ as stated in the following Lemma.
\begin{lem}
For any $k \ge 1$ it holds
\begin{equation}\label{Bkcontrol2}
B^k1(t)\le (g(t))^ku_k^*(t)
\end{equation}
\end{lem}
\begin{proof}
	Let us first observe that $B1(t)=g(t)U_\Phi(t)=g(t)u_1^*(t)$. Let us suppose Equation \eqref{Bkcontrol2} holds for some $k \ge 1$. Being $g(t)$ non-decreasing, we have
	\begin{multline*}
	B^{k+1}1(t)=B(B^k1(t))(t)\le B((g(t))^ku_k^*(t))(t)=\\=g(t)\int_0^tu_\Phi(t-s)(g(s))^ku_k^*(s)ds\le (g(t))^{k+1}u_{k+1}^*(t),
	\end{multline*}
	concluding the proof.
\end{proof}
Now let us focus on series defined via the operator $B$ and its powers.
\begin{lem}\label{lemseries}
	Let $f \in L^1([0,T])$. Then $\sum_{k=1}^{+\infty}B^kf(t)$ normally converges for any $t \in [0,T]$.
\end{lem}
\begin{proof}
	Let us first observe that without loss of generality we can suppose $f(t)\ge 0$. By Lemma \ref{PEBk}, we have
	\begin{align*}
	\sum_{k=1}^{+\infty}B^kf(t)&\le \sum_{k=1}^{+\infty}\frac{(C\Gamma(\beta)g(t))^k}{\Gamma(k\beta)}\int_0^t(t-s)^{k\beta-1}f(s)ds\\
	&\le \sum_{k=1}^{+\infty}\frac{(C\Gamma(\beta)g(t))^k}{\Gamma(k\beta)}T^{k\beta-1}\int_0^tf(s)ds\\
	& \le \Norm{f}{L^1([0,T])}\sum_{k=1}^{+\infty}\frac{(C\Gamma(\beta)g(T))^k}{\Gamma(k\beta)}T^{k\beta-1}\\
	&=\frac{C\beta\Gamma(\beta)g(T)\Norm{f}{L^1([0,T])}}{T^{1-\beta}}\sum_{k=1}^{+\infty}\frac{k(C\Gamma(\beta)g(T)T^\beta)^{k-1}}{\Gamma(k\beta+1)}\\
	&=\frac{C\beta\Gamma(\beta)g(T)\Norm{f}{L^1([0,T])}}{T^{1-\beta}}E'_\beta(C\beta\Gamma(\beta)g(T)T^\beta),
	\end{align*}
	where $E_\beta(z)$ is the one-parameter Mittag-Leffler defined in Remark \ref{MLrmk}.
\end{proof}
\begin{rmk}
	Let us remark that the previous Lemma implies that \linebreak $\lim_{k \to +\infty}B^kf(t)=0$ uniformly in $[0,T]$.
\end{rmk}
Let us show a last technical Lemma.
\begin{lem}\label{confr2}
	Let $f_1,f_2 \in L^1(0,T)$ with $f_1,f_2 \ge 0$ and $f_1$ non-decreasing. Then
	\begin{equation}\label{eqconfr2}
	B^k(f_1f_2)(t)\le f_1(t)B^kf_2(t)
	\end{equation}
	for any $t \in [0,T]$.
\end{lem}
\begin{proof}
	As $k=1$ we have, being $f_1$ non-decreasing,
	\begin{equation}\label{eqconfr3}
	B(f_1f_2)(t)=g(t)\int_0^t u(t-s)f_1(s)f_2(s)ds\le g(t)f_1(t) \int_0^t u(t-s)f_2(s)ds=f_1(t)B(f_2)(t).
	\end{equation}
	Now let us suppose Equation \eqref{eqconfr2} holds for some $k \ge 1$. Then it holds
	\begin{equation*}
	B^{k+1}(f_1f_2)(t)=B(B^k(f_1f_2))(t)\le B(f_1B^kf_2)(t)\le f_1(t)B^{k+1}f_2(t),
	\end{equation*}
	concluding the proof.
	\end{proof}
\subsection{Proof of Theorem \ref{thm:Gron}}\label{subsectionproof}
Let us first rewrite inequality \eqref{intineq} in terms of $B$, i.e.
\begin{equation}\label{est1}
x(t)\le a(t)+Bx(t), \ t \in [0,T].
\end{equation}
We want to show that for any $n \ge 1$
\begin{equation}\label{est3}
x(t)\le \sum_{k=0}^{n-1}B^ka(t)+B^nx(t), \ t \in [0,T].
\end{equation}
As $n=1$ this is actually Equation \eqref{est1}. Let us suppose inequality \eqref{est3} holds for some $n \ge 1$. Applying $B^n$ on both sides of inequality \eqref{est1} we have
\begin{equation*}
B^n x(t)\le B^n a(t)+B^{n+1}x(t), \ t \in [0,T].
\end{equation*}
Using last relation in inequality \eqref{est3} we get
\begin{equation}
x(t)\le \sum_{k=0}^{n}B^ka(t)+B^{n+1}x(t), \ t \in [0,T].
\end{equation}
Hence, by induction, we know that inequality \eqref{est3} holds for any $n \ge 1$. In particular we can take the limit as $n \to +\infty$ in \eqref{est3}, recalling that $\lim_{n \to +\infty}B^nx(t)=0$, to achieve
\begin{equation}\label{est4}
x(t)\le a(t)+\sum_{k=1}^{+\infty}B^ka(t),
\end{equation}
proving the first part of the Theorem.\\
Concerning the second part, we have, by Lemma \ref{PEBk} and monotone convergence theorem,
\begin{align*}
	\sum_{k=1}^{+\infty}B^ka(t)&\le \int_0^t\sum_{k=1}^{+\infty}\frac{(C\Gamma(\beta)g(t))^k}{\Gamma(k\beta)}(t-s)^{k\beta-1}a(s)ds\\
	&= C\beta \Gamma(\beta)g(t)\int_0^t(t-s)^{\beta-1}\left(\sum_{k=1}^{+\infty}\frac{k(C\Gamma(\beta)g(t)(t-s)^\beta)^{k-1}}{\Gamma(k\beta+1)}\right)a(s)ds\\
	&\le C\Gamma(\beta+1)g(t)\int_0^t E'_{\beta}(C\Gamma(\beta)g(t)(t-s))(t-s)^{\beta-1}a(s)ds.
\end{align*}
Concerning the third part of the Theorem, let us observe that, since $a$ is non-decreasing,
\begin{align*}
	x(t)\le \sum_{k=0}^{+\infty}B^ka(t)\le a(t)\sum_{k=0}^{+\infty}(g(T))^ku_k^*(t)=a(t)\fe_\Phi(t,g(T)),
	\end{align*}
	where last equality follows from Theorem \ref{thm:series}.\qed
\section{Consequences of the generalized Gr\"onwall inequality}\label{Sec6}
In this section we focus on some consequences of the generalized Gr\"onwall inequality proved in the previous section.
\subsection{Continuous dependence on the initial data}
Let us first consider continuous dependence on the initial data. Fix a Banach space $(X,|\cdot|)$, a time horizon $T>0$ and a function $F:[0,T]\times X \to X$ satisfying hypotheses $A1$ and $A2$ of Theorem \ref{ex}. We would like to study the continuity in $C(J',X)$ for some time interval $J' \subseteq J$ of the solutions of the Cauchy problem \eqref{Cprob} as the initial datum $f_0 \in X$ varies. To do this, we first need to show that, for fixed $f_0 \in X$, for any initial data $\widetilde{f}_0$ taken in a suitable neighbourhood of $f_0$ there exists a unique solution $f:J' \to X$ of \eqref{Cprob} where $J'$ is independent of $\widetilde{f}_0$.
\begin{prop}\label{prop:simex1}
	Let $\Phi \in \SBF$ satisfy Assumption \ref{assp} and $F:[0,T]\times X \to X$ satisfy hypotheses $A1$ and $A2$. Fix $f_0 \in X$ and $R>0$. Then there exists $T' \in (0,T]$ such that for any $\widetilde{f}_0 \in B_{1}(f_0)$ there exists a unique solution $f \in C(J',B_R(\widetilde{f}_0))$ of \eqref{Cprob} admitting $\widetilde{f}_0$ as initial datum, where $J'=[0,T']$.
\end{prop}
\begin{proof}
	To prove this statement we need to show that we can consider $T'$ in Lemma \ref{lemT1} to be equal for any $\widetilde{f}_0 \in B_{1}(f_0)$. This time let us define $\widetilde{R}=R+|f_0|+1$ and observe that if $\widetilde{f}_0 \in B_{1}(f_0)$ and $x \in B_R(\widetilde{f}_0)$ then we have 
	\begin{equation*}
	|x|\le |x-\widetilde{f}_0|+|\widetilde{f}_0-f_0|+|f_0|<R+1+|f_0|.
	\end{equation*}
	Choosing $T'>0$ such that $C_{\widetilde{R}}U(T')<R$ we conclude the proof.
\end{proof}
Let us fix $f_0 \in X$, $R>0$ and $T'>0$ as in Proposition \ref{prop:simex1}. Define the function
\begin{equation*}
\Psi:\widetilde{f}_0\in B_1(f_0) \to \Psi(\cdot;\widetilde{f}_0) \in C(J',X)
\end{equation*}  
where, for fixed $\widetilde{f}_0 \in B_1(f_0)$, $t \in [0,T] \mapsto \Psi(t;\widetilde{f}_0) \in X$ is solution of the Cauchy problem \eqref{Cprob} with initial datum $\widetilde{f}_0$. We want to show that $\Psi$ is continuous in $f_0$.
\begin{prop}
	Let $\Phi \in \SBF$ satisfy Assumption \ref{assp} and $F:[0,T]\times X \to X$ satisfy hypotheses $A1$ and $A2$. Fix $f_0 \in X$ and $R>0$. Fix $T'>0$ as in Proposition \ref{prop:simex1} and define $\Psi$ as mentioned before. Then $\Psi$ is continuous in $f_0$.
\end{prop}
\begin{proof}
	Let us define $h(t;\widetilde{f}_0)=|\Psi(t;\widetilde{f}_0)-\Psi(t;f_0)|$. Since $\Psi(t;\widetilde{f}_0)$ is solution of \eqref{Cprob}, we have, defining $\widetilde{R}=R+1+|f_0|$,
	\begin{align}\label{est11}
	\begin{split}
	h(t,\widetilde{f}_0)&\le |\widetilde{f}_0-f_0|+\int_0^tu_\Phi(t-s)\left|F(s,\Psi(s;\widetilde{f_0}))-F(s,\Psi(s;f_0))\right|\\
	&\le |\widetilde{f}_0-f_0|+L_{\widetilde{R}}\int_0^tu_\Phi(t-s)h(s,\widetilde{f}_0)ds.
	\end{split}
	\end{align}
	By the third part of Theorem \ref{thm:Gron} we have
	\begin{equation*}
	h(t;\widetilde{f_0})\le |\widetilde{f}_0-f_0|\fe_\Phi(t;L_{\widetilde{R}})\le |\widetilde{f}_0-f_0|\fe_\Phi(T';L_{\widetilde{R}}).
	\end{equation*} 
	Taking the maximum as $t \in [0,T']$ we conclude the proof.
\end{proof}
\begin{rmk}
	With a similar proof, we have that $\Psi$ is Lipschitz continuous in $B_1(f_0)$. 
\end{rmk}
\subsection{Continuous dependence on a parameter}
A similar approach can be used to show continuous dependence on a parameter. Fix a Banach space $(X,|\cdot|)$, a locally compact metric space $(V,d)$ and a function $F:[0,T]\times X \times V \to X$. Now we want to focus on the parametric Cauchy problem
\begin{equation}\label{Cprobpar}
\begin{cases}
\partial_t^\Phi f(t)=F(t,f(t);v) & t \in [0,T] \mbox{ a.e.}\\
f(0)=f_0
\end{cases}
\end{equation}
As before, we need to show that there exists a common time-interval of existence.
\begin{prop}\label{prop:simex2}
	Let $\Phi \in \SBF$ satisfy Assumption \ref{assp} and $F:[0,T]\times X \times V \to X$ such that
	\begin{itemize}
		\item[$A1'$] For any compact set $K \subseteq V$ and any ball $B_R$ in $X$ there exists a constant $C_{R,K}>0$ such that $|F(t,x;v)|\le C_{R,K}$ for almost any $t \in J$, any $x \in B_R$ and any $v \in K$;
		\item[$A2'$] For any compact set $K \subseteq V$ and any ball $B_R$ in $X$ there exists a constant $L_{R,K}>0$ such that $|F(t,x;v)-F(t,z;v)|\le L_{R,K}|x-z|$ for almost any $t \in J$, any $x,z \in B_R$ and any $v \in K$.
	\end{itemize}
	Fix $f_0 \in X$, $v_0 \in V$ and $R>0$. Then there exists $T' \in (0,T]$ and a compact neighbourhood $K$ of $v_0$ such that for any fixed $v \in K$ there exists a unique solution $f \in C(J',B_R(f_0))$ of \eqref{Cprobpar}, where $J'=[0,T']$.
\end{prop}
\begin{proof}
	Being $V$ a locally compact metric space, there exists a compact neighbourhood $K$ of $v_0$. Hence, from hypotheses $A1'$ and $A2'$, we have that for fixed $v \in K$, $F(\cdot,\cdot;v)$ satisfies hypotheses $A1$ and $A2$ of Theorem \ref{ex} with constants that are independent of the choice of $v \in K$. Hence, for fixed $R>0$, the choice of $T'$ is independent of $v \in K$.
\end{proof}
Now let us fix $v_0 \in V$, consider $K$ and $T'$ as in Proposition \ref{prop:simex2} and define the function $\Psi:v \in K \mapsto \Psi(\cdot;v) \in C(J',X)$ such that, for any fixed $v \in K$, $\Psi(\cdot;v)$ is solution of \eqref{Cprobpar}. We want to show that $\Psi$ is continuous in $v_0$ under some additional hypotheses on $F$.
\begin{prop}\label{prop:condep}
	Let $\Phi \in \SBF$ satisfy Assumption \ref{assp} and $F:[0,T]\times X \times V \to X$ satisfy $A1'$, $A2'$ and
	\begin{itemize}
		\item[$A3'$] For any compact set $K \subseteq V$ and any ball $B_R$ in $X$ there exists a constant $\widetilde{L}_{R,K}>0$ such that $|F(t,x;v_1)-F(t,x;v_2)|\le \widetilde{L}_{R,K}d(v_1,v_2)$ for almost any $t \in J$, any $x \in B_R$ and any $v_1,v_2 \in K$.
	\end{itemize}
	Fix $f_0 \in X$, $v_0 \in V$ and $R>0$. Consider a compact neighbourhood $K$ of $v_0$ and $T'>0$ as in Proposition \ref{prop:simex2}. Define $\Psi$ as stated before. Then $\Psi$ is continuous in $v_0$.
\end{prop}
\begin{proof}
	Let us define $h(t;v)=|\Psi(t;v)-\Psi(t;v_0)|$. Being $\Psi(t;v)$ and $\Psi(t;v_0)$ solutions of \eqref{Cprobpar} we have
	\begin{align*}
	h(t;v)&\le\int_0^tu_\Phi(t-s)|F(s,\Psi(s;v);v)-F(s,\Psi(s;v_0);v_0)|ds\\
	&\le \int_0^tu_\Phi(t-s)|F(s,\Psi(s;v);v)-F(s,\Psi(s;v);v_0)|ds\\
	&+\int_0^tu_\Phi(t-s)|F(s,\Psi(s;v);v_0)-F(s,\Psi(s;v_0);v_0)|ds\\
	&\le \widetilde{L}_{R,K}d(v_1,v_2)U_\Phi(t)+L_{R,K}\int_0^tu_\Phi(t-s)h(s;v)ds.
	\end{align*}
	Hence, by Theorem \ref{thm:Gron}, we have
	\begin{equation*}
	h(t;v)\le \widetilde{L}_{R,K}d(v_1,v_2)U_\Phi(t) \fe_\Phi(t;L_{R,K})\le \widetilde{L}_{R,K}d(v_1,v_2)U_\Phi(T) \fe_\Phi(T;L_{R,K}),
	\end{equation*} 
	concluding the proof.
\end{proof}
\begin{rmk}
	As in the case of the initial datum, $\Psi$ is actually Lipschitz-continuous in $K$.
\end{rmk}
\subsection{Global uniqueness}
In Corollary \ref{corollaryexist} we have shown that if a non-homogeneous linear problem admits a global solution, then such solution is continuous. Now that we have a generalized Gr\"onwall inequality, we can extend this result to the non-linear case.
\begin{prop}
	Let $\Phi\in \SBF$ satisfy Assumption \ref{assp} with $\nu_\Phi$ absolutely continuous, $f_0 \in X$ and $F:[0,+\infty) \times X \to X$ satisfying
	\begin{itemize}
		\item[$A2_{\rm loc}$] For any $T>0$ and any ball $B_R$ in $X$ there exists a constant $L_R>0$ such that $|F(t,x)-F(t,z)|\le L_R|x-z|$ for almost any $t \in [0,T]$ and any $x,z \in B_R$.
	\end{itemize}
	Then, if the Cauchy problem \eqref{Cprob} admits a solution in $C([0,+\infty),X)$, it is unique.
\end{prop}
\begin{proof}
	Arguing as in Corollary \ref{corollaryexist}, let us suppose there are two solutions $f_1(t)$ and $f_2(t)$ of \eqref{Cprob} defined on $[0,+\infty)$. By Theorem \ref{ex} we already know there exists $T'>0$ such that $f_1(t)=f_2(t)$ for any $t \in [0,T']$. Let us then consider $T_*=\sup\{t>0: f_1(t)=f_2(t)\}$. We want to show that $T_*=+\infty$. Let us argue by contradiction: suppose that $T_*<+\infty$. Arguing as in Corollary \ref{corollaryexist}, we have
	\begin{align*}
	\der{}{t}\int_{T_*}^{t}&\bar{\nu}_\Phi(t-s)(f_1(s)-f_*)ds=F(t,f_1(t))\\
	&+\int_0^{T_*} \nu_\Phi(t-s)(f_1(s)-f_0)ds -(f_*-f_0)\overline{\nu}_\Phi(t-T_*).
	\end{align*}
	The same relation holds for $f_2(t)$. Define $g(t)=f_1(t)+f_2(t)$. Observing that $g(t)\equiv 0$ for any $t \in [0,T_*]$, we have
	\begin{align*}
	\der{}{t}\int_{T_*}^{t}&\bar{\nu}_\Phi(t-s)g(s)ds=F(t,f_1(t))-F(t,f_2(t))
	\end{align*}
	and then, setting $w=s-T_*$,
	\begin{align*}
	\der{}{t}\int_{0}^{t-T_*}&\bar{\nu}_\Phi(t-T_*-w)g(T_*+w)dw=F(t,f_1(t))-F(t,f_2(t)).
	\end{align*}
	Defining $\widetilde{g}(t)=g(T_*+t)$, $\widetilde{f}_i(t)=f_i(T_*+t)$, $\widetilde{F}(t,x)=F(T_*+t,x)$ and applying $\cI_t^\Phi$ on both sides we have
	\begin{align*}
	\widetilde{g}(t-T_*)=\cI_t^\Phi(\widetilde{F}(\cdot,\widetilde{f}_1(\cdot))-\widetilde{F}(\cdot,\widetilde{f}_2(\cdot)))(t-T_*).
	\end{align*}
	Now define $h(t)=|\widetilde{g}(t)|$ and substitute $t$ to $t-T_*$. Then we have
	\begin{align*}
	h(t)\le \cI_t^\Phi(|\widetilde{F}(\cdot,\widetilde{f}_1(\cdot))-\widetilde{F}(\cdot,\widetilde{f}_2(\cdot))|)(t).
	\end{align*}
	Fix $T'>0$. Being $f_1,f_2 \in C([0,+\infty),X)$, there exists a ball $B_R$ such that $f_1(t),f_2(t) \in B_R$ for any $t \in [0,T_*+T']$. By hypothesis $A2_{\rm loc}$ we obtain
	\begin{equation*}
	|\widetilde{F}(t,\widetilde{f}_1(t))-\widetilde{F}(t,\widetilde{f}_2(t))|\le L_R h(t)
	\end{equation*}
	for $t \in [0,T']$	and then we have
	\begin{align*}
	h(t)\le L_R \cI_t^\Phi h(t) \ t \in [0,T'].
	\end{align*}
	By Theorem \ref{thm:Gron} we have $h(t)\equiv 0$ for $t \in [0,T']$. However, this implies $f_1(t)=f_2(t)$ for $t \in [T_*,T_*+T']$, which is absurd. Hence $T_*=+\infty$.
\end{proof}

\bibliographystyle{abbrv}
\bibliography{bib}
\end{document}